\documentclass{amsart}
\usepackage{graphicx}

\usepackage[
    a4paper,
    left=1.5cm,
    right=1.5cm,
    top=2.5cm,
    bottom=2.5cm
]{geometry}
\usepackage{amsfonts,amsthm,latexsym,amsmath,amssymb,amscd,amsmath,amsthm,mathrsfs,epsf,stmaryrd}
\usepackage{enumerate}
\usepackage[utf8]{inputenc}

\usepackage[dvipsnames]{xcolor}
\usepackage{tikz}
\usetikzlibrary{arrows,matrix}
\usepackage{blkarray}
\usepackage[linewidth=1pt]{mdframed}
\usepackage{mathtools}

\usepackage{subcaption}
\usepackage{caption}

\usepackage{algorithm}
\usepackage{algorithmicx,algpseudocode}
\algrenewcommand{\algorithmiccomment}[1]{\hspace{-19pt} {\normalsize\color{black}$\triangleright\;$\textit{#1}}}

\usepackage{enumitem}

\usepackage{aliascnt}
\usepackage{hyperref}
\hypersetup{
    colorlinks=true,
    linkcolor=MidnightBlue,
    citecolor=Green
}
\usepackage[nameinlink,noabbrev]{cleveref}

\theoremstyle{plain}

\newaliascnt{lemma}{theorem}
\newtheorem{lemma}[lemma]{Lemma}
\aliascntresetthe{lemma}

\newaliascnt{proposition}{theorem}
\newtheorem{proposition}[proposition]{Proposition}
\aliascntresetthe{proposition}

\newaliascnt{corollary}{theorem}
\newtheorem{corollary}[corollary]{Corollary}
\aliascntresetthe{corollary}

\newaliascnt{definition}{theorem}
\newtheorem{definition}[definition]{Definition}
\aliascntresetthe{definition}

\newaliascnt{example}{theorem}
\newtheorem{example}[example]{Example}
\aliascntresetthe{example}

\newaliascnt{remark}{theorem}
\newtheorem{remark}[remark]{Remark}
\aliascntresetthe{remark}

\newaliascnt{question}{theorem}

\aliascntresetthe{question}

\newaliascnt{alg}{theorem}

\aliascntresetthe{alg}

\crefname{theorem}{Theorem}{Theorems}
\Crefname{theorem}{Theorem}{Theorems}

\crefname{lemma}{Lemma}{Lemmas}
\Crefname{lemma}{Lemma}{Lemmas}

\crefname{proposition}{Proposition}{Propositions}
\Crefname{proposition}{Proposition}{Propositions}

\crefname{corollary}{Corollary}{Corollaries}
\Crefname{corollary}{Corollary}{Corollaries}

\crefname{definition}{Definition}{Definitions}
\Crefname{definition}{Definition}{Definitions}

\crefname{example}{Example}{Examples}
\Crefname{example}{Example}{Examples}

\crefname{remark}{Remark}{Remarks}
\Crefname{remark}{Remark}{Remarks}

\crefname{question}{Question}{Questions}
\Crefname{question}{Question}{Questions}

\crefname{alg}{Algorithm}{Algorithms}
\Crefname{alg}{Algorithm}{Algorithms}

\crefname{claim}{Claim}{Claims}
\Crefname{claim}{Claim}{Claims}

\crefname{equation}{eq.}{eqs.}
\Crefname{equation}{Eq.}{Eqs.}

\renewcommand{\k}{\Bbbk}
\newcommand{\K}{\mathbb{K}}
\newcommand{\Char}{\textnormal{char}\,}

\renewcommand{\S}{\mathcal{S}}

\newcommand{\N}{\mathbb{N}}

\newcommand{\Hom}{\textnormal{Hom}}

\newcommand{\Ann}{\textnormal{Ann}}
\newcommand{\len}{\textnormal{len}\,}
\renewcommand{\P}{\mathbb{P}}
\def\a{\alpha}
\def\b{\beta}
\def\c{\gamma}
\newcommand{\ideal}[1]{\mathopen{}\bigl( #1 \bigr)\mathclose{}}

\newcommand{\contract}{\mathbin{\mathpalette\dointprod\relax}}
\newcommand{\dointprod}[2]{%
  \raisebox{\depth}{\scalebox{1}[-1]{$#1\lnot$}}}

\usepackage{hyperref}
\def\\k{{\mathsf{k}}}
\def\P{{\mathbb{P}}}

\def\<{\left<}
\def\>{\right>}

\def\R{\mathcal{R}}

\def\Z{\mathfrak{Z}}

\def\PGL{\mathrm{PGL}}
\def\GL{\mathrm{GL}}

\def\Ann{\mathrm{Ann}}

\numberwithin{equation}{section}

\begin{document}

\title{Determinantal computation of minimal local GADs}

\author[O. Reig Fité]{Oriol Reig Fité}
\address{Universit\`a di Trento, Via Sommarive, 14 - 38123 Povo (Trento), Italy}
\email{oriol.reigfite@unitn.it}

\author[D. Taufer]{Daniele Taufer}
\address{KU Leuven, Celestijnenlaan 200A, 3001 Leuven, Belgium}
\curraddr{Max Planck Institute for Mathematics in the Sciences, Inselstraße 22, 04103 Leipzig, Germany}
\email{daniele.taufer@gmail.com}

\subjclass[2020]{Primary: 13H10, 14N07, 15A69}


\begin{abstract}
We study local generalized additive decompositions (GADs) of homogeneous polynomials and their associated punctual schemes through their local inverse systems.
We verify that their construction and algebraic properties are independent of the chosen apolarity action.
We propose a determinantal method for computing minimal local GADs by minimizing the rank of a symbolic inverse system.
When the locus of minimal supports is finite, this procedure provides a practical tool to determine all minimal local decompositions without tensor extensions.
We prove that this finiteness is guaranteed whenever the local GAD-rank of the form does not exceed its degree.
We analyze both generic and special cases, provide computational evidence assessing the impact of different choices for minors in the determinantal algorithm, and compare our approach with existing algorithms for local apolar schemes.
\end{abstract}

\maketitle

\section{Introduction}
The study of additive decompositions of homogeneous polynomials is a central topic in commutative algebra and applied algebraic geometry, with far-reaching applications to symmetric tensor decomposition.
The classical \emph{Waring decomposition problem}, namely expressing a form as a minimal sum of powers of linear forms, has been extensively investigated in connection with secant varieties, catalecticant matrices, and apolarity theory \cite{BB12,BB14,BK24,IK99,L11}.
From a scheme-theoretic perspective, the solution to the Waring problem for a given form $F$ is equivalent to detecting minimal \emph{reduced} zero-dimensional schemes that are apolar to $F$.
Despite its theoretical interest and multitude of applications, this remains a highly challenging problem: computing such minimal configurations is, in general, difficult without additional structural assumptions \cite{BT20,BCMT10}, even when such minima are unique (\emph{identifiable}) \cite{CM23,COV14}.

Over the past decades, several generalizations of this problem have been considered, capturing more sophisticated intrinsic properties of a given form.
The minimal length of such decompositions, known as their \emph{ranks} \cite{BBM14}, reflects different aspects of the complexity of the considered form and is deeply intertwined with 
the structure of the associated apolar algebras.
In particular, the Castelnuovo-Mumford regularity of such algebras has proven to be a key measure for the computational cost of effectively producing such minimal decompositions \cite{BMT25,BOT24,GKT25,LKS25}.

In this work, we tackle the problem of studying \emph{generalized additive decompositions} \cite{IK99} of a given degree-$d$ form $F$, namely $F = \sum_{i=1}^s \omega_i \ell_i^{d-k_i}$, where $\omega_i$ and $\ell_i$ are forms of degree $k_i$ and $1$, respectively, and $0\leq k_i\leq d$.
When such decompositions are also \emph{local}, i.e., $s=1$, they are shortened as \emph{local GADs} and consist of factorizations of $F$ as
\[ F = \omega \ell^{d-k}. \]
Unlike global GADs, local decompositions are uniquely determined by their support $\ell$, hence the space of local GADs of a given form $F \in \k[x_0,\dots,x_n]$ is parametrized by $\P^n(\k)$.
Our goal is to effectively determine those decompositions having minimal algebraic complexity, measured by the length of the naturally associated apolar punctual scheme \cite{BJMR18,BOT24}.
More generally, we investigate the stratification of the parameter space $\P^n(\k)$ by the GAD size of a prescribed form, namely, we aim at a complete description of the locus of supports evincing any given GAD size.

Geometrically, a local GAD $F = \omega \ell^{d-k}$ expresses the membership of $[F] \in \P^{{n+d \choose d}-1}$ in the $k$-osculating space to the Veronese variety at $[\ell^d]$ \cite{BMT25,BCGI07}. 
However, minimizing local GADs does \emph{not} correspond to computing the minimal osculation degree $k$ needed to span $F$. 
In fact, the minimal supports $\ell$ might well involve $k=d$ (even when $F$ is reducible), and the locus of such minimal $\ell$'s may have a positive dimension.
This rich geometric behavior reflects the intricate structure of non-reduced points that are \emph{apolar} to a given form $F$, which is governed by the structure of (local) Artinian Gorenstein algebras and their \emph{inverse systems}.
These local objects are essential for understanding the \emph{cactus rank} \cite{BB14}, also known as \emph{minimal scheme length} \cite{IK99}, of forms.
Indeed, this rank can always be computed from GADs of certain primitives (i.e., antiderivatives) of $F$, and in the local case, it coincides with the size of a minimal local GAD if the socle degree of a minimal apolar scheme does not exceed the degree of $F$ \cite{BOT24}.
There exist systematic methods to test and certify \emph{global} minimality \cite{BMT25,BT20}, which have recently also been adapted to address the local decompositions considered in the present manuscript \cite{BF25}.
However, they rely on large and highly structured parameter spaces, making explicit computations rapidly infeasible even for moderately sized examples.
Moreover, with the current approaches, the number of unknown moments and the staircases to test grows rapidly when addressing non-minimal GADs, which is a major obstruction to classifying the GAD size strata of any given form.

\subsection*{Novel contributions}
We connect different notions of apolarity (differentiation, contraction, and $\star$-action) and their corresponding constructions of punctual schemes evinced by GADs, as they appear in recent literature \cite{BMT25,BJMR18,BOT24}.
In particular, we highlight the explicit correspondence between the inverse systems defined by the respective dual generators (\Cref{prop:equivgenerators}), which is compatible with both the derivation action (in the polynomial-exponential setting) and the contraction action (in the divided powers framework).

We then adopt the latter viewpoint, and propose a new exact method to compute minimal local GADs of a given form in $\k[x_0, \dots,x_n]$, which is based on minimizing the dimension of the corresponding symbolic inverse system (\Cref{algo:localgad}).
This procedure is particularly effective when the set of \emph{minimal supports} is finite, as we collect determinantal relations from the \emph{inverse system matrix} (\Cref{defn:IFell}) of a form $F$ at a symbolic linear form $\ell \in \k[\c_1, \dots,\c_n][x_0, \dots, x_n]$, until they define a zero-dimensional ideal in the parameter space $\k[\c]$.
This allows the computation of minimal local GADs without constructing the whole minor ideals of Hankel matrices, which would be prohibitive even in moderate cases.

A key aspect in this process is the selection of suitable minors.
When the size of a minimal local GAD $F$ does not exceed its degree, we prove that a specific choice of minors implies the existence of a finite number of minimal supports (\cref{prop:FinitelyManyMinimalSupports}), which applies similarly to local cactus (\Cref{cor:localcactus}).
For the general case, we propose and test different practical strategies, and we observe that the best performance is often achieved by \emph{following the contraction chains} (\Cref{sec:RankMinimiz}).

We establish lower and upper bounds for the local GAD-rank (\Cref{sec:boundslocalGADrank}), providing examples that match these extremal values.
We observe that generic forms may not achieve this maximal GAD-rank (\Cref{ex:genericnotmax}), and we apply the proposed determinantal method to understand the structure of \emph{every} local GAD for certain given forms (\Cref{ex:runningexample,ex:InfiniteMinimalLocus,ex:funnymonomial}).
Moreover, these examples show that certain forms admit an infinite locus of supports yielding a given, sub-maximal, local GAD size, which can even agree with the local GAD-rank.
This analysis is much finer than what was achievable with previous methods, highlighting the effectiveness of the proposed approach for studying local GADs.
Moreover, it does not involve any tensor extensions, and its complexity appears to rely only on the local GAD-rank and on the finiteness of the minimal loci, rather than the CM-regularity of the associated schemes.
Finally, we connect the considered inverse system matrix to the classical Hankel and catalecticant matrices associated to a given form (\Cref{sec:comparison}).

\subsection*{Paper organization}
The paper is organized as follows.
In \Cref{sec:Preliminaries}, we review the necessary background on GADs, apolarity, and inverse systems.
In \Cref{sec:minLocalGAD}, we introduce the symbolic inverse system and apply it to understanding and computing minimal local GADs.
In \Cref{sec:comparison}, we compare our method with existing algorithms for local apolar schemes.
Finally, we propose directions for future research in \Cref{sec:conclusion}.


\section{Preliminaries} \label{sec:Preliminaries}

\subsection{Notation}

In this paper, $n \in \N_{>0}$ will be a fixed positive integer, and $\k$ will be a field.
We assume that the characteristic of $\k$ is zero or strictly larger than the degree of any form considered, and we will further assume that all the considered schemes $\Z \subset \P^n(\k)$ have $\k$-rational support.
We will also employ the symbolic ring $\K = \k[\c]$, where $\c = (\c_1, \dots, \c_n)$ are parameters to be specialized in $\k$.

\subsection{Apolarity actions}

Let $\S = \oplus_{d \in \N} \S_d = \k[x_0, \dots, x_n]$ and $\R = \oplus_{d \in \N} \R_d = \k[y_0, \dots, y_n]$ be standard graded multivariate polynomial $\k$-algebras.
The \emph{apolarity} action of $\R$ on $\S$ is defined, in the literature, by either \emph{derivation} or \emph{contraction}, namely by $\k$-linearity as
\[ y^{\b} \circ x^{\a} = \begin{cases}
    \frac{\a!}{(\a-\b)!} x^{\a-\b} & \textnormal{if } \a \geq \b,\\
    0 & \textnormal{otherwise},
\end{cases} \quad \textnormal{(derivation),} \qquad 
y^{\b} \contract x^{\a} = \begin{cases}
    x^{\a-\b} & \textnormal{if } \a \geq \b,\\
    0 & \textnormal{otherwise},
\end{cases} \quad \textnormal{(contraction).}
\]
Let $F_{\rm dp} = \sum_{|\a|=d} \a!F_{\a}x^{\a}$ be the \emph{divided power} representation of $F = \sum_{|\a|=d} F_{\a}x^{\a} \in \S_d$.
We denote by $\S_{\rm dp}$ the divided power algebra of $\S$ \cite[§9]{G96}.
The divided power operator ${\rm dp} : \S \to \S_{\rm dp},\, F \mapsto F_{\rm dp}$ commutes with the $\R$-module structure on $\S$, given by differentiation, and that of $\S_{\rm dp}$, given by contraction, since we have
\begin{equation} \label{eq:circvscontract}
( y^{\b} \circ x^{\a} )_{\rm dp} = y^{\b} \contract ( x^{\a} )_{\rm dp}.
\end{equation}
When $\Char \k = 0$, it is well-known that ${\rm dp}$ is an isomorphism of $\R$-modules \cite[Thm. 9.5]{G96}.

We will also consider linear forms $\ell = x_0 + \xi_1 x_1 + \dots + \xi_n x_n \in \S_1$ corresponding to points $\xi = [1:\xi_1:\dots:\xi_n] \in \P^n(\k)$ in the first affine chart (namely, the points with non-zero first coordinate).
This assumption is not restrictive: up to a generic change of variables in $\PGL_{n+1}(\k)$, we can always assume our punctual schemes to lie in this chart.
We will identify the coordinate ring of this affine chart with $\S' = \k[x_1,\ldots,x_n]$, and similarly $\R'=\k[y_1,\ldots,y_n]$.
Again, $\S'$ can be given $\R'$-module structures by derivation and contraction, as above.

A related action, which is also referred to as apolar action in the literature, is the $\star$-\emph{action} of $\R$ on $\R^* = \Hom_{\k}(\R,\k)$: for any $g \in \R$ and $\phi \in \R^*$, we define the linear operator
\[ g \star \phi: \R \to \k, \quad p \mapsto \phi(gp) \quad \textnormal{($\star$-action)}. \]

\subsection{GADs and natural apolar schemes} \label{sec:GADsplus}

A \emph{generalized additive decomposition} (shortened as \emph{GAD}) of $G \in \S_d$ is a decomposition as
\begin{equation} \label{eq:GAD}
    G = \sum_{i=1}^s \omega_i \ell_i^{d-k_i},
\end{equation} 
where $0\leq k_i\leq d$, the linear forms $\ell_i\in \S_1$ are non-proportional, i.e., they correspond to distinct points of $\P^n(\k)$, and they do not divide the corresponding $\omega_i \in \S_{k_i}$.
They were first introduced in \cite[Def. 1.30]{IK99}, where they were called \emph{reduced GADs} if $\ell_i \nmid \omega_i$. In this paper, we will make no such distinction: all GADs will be reduced by definition, as in \cite{BOT24}.
We will also always identify projectively equivalent forms, i.e., forms that differ by a nonzero scalar, and consider \emph{finiteness} and \emph{uniqueness} only up to this identification.

Each GAD canonically \emph{evinces} a punctual scheme $\Z \subset \P^n(\k)$, supported at the points corresponding to the $\ell_i$'s, and whose schematic structure is determined by the corresponding $\omega_i$ \cite{BOT24}.
Such a $\Z$ is also known as the scheme \emph{associated with the GAD} \cite{BMT25}.
For a local GAD $F = \omega \ell^{d-k}$, this scheme is known as the \emph{natural apolar scheme of $F$ at $\ell$} \cite{BJMR18}, and it is explicitly constructed as in \cite[Algorithm 1]{BOT24}. For $\ell = x_0$, the construction is the following:
\begin{enumerate}
    \item Consider the de-homogenization $f_{\ell} \in \k[x_1, \dots,x_n]$ of $F_{\rm dp}$ by the linear form $\ell$,
    \item compute the affine ideal $\Ann^{\contract}( f_{\ell} ) = \{ g \in \R' \ : \ g \contract f_{\ell} = 0 \} \subset \R'$, 
    \item homogenize $\Ann^{\contract}( f_{\ell} )$ in $\R$, using the dual variable $\ell^*$ as homogenizing variable. 
\end{enumerate}
For a general GAD, one applies the construction above to each summand, and takes the union of the local schemes obtained: if we denote the natural apolar scheme to $\omega \ell^{d-k}$ at $\ell$ by $\Z_{\omega \ell^{d-k}, \ell}$, then the GAD $\sum_{i=1}^s \omega_i \ell_i^{d-k_i}$ evinces the scheme 
\[\Z=\bigcup_{i=1}^s \Z_{\omega_i \ell_i^{d-k_i}, \ell_i}.\]
The homogeneous ideal constructed this way is \emph{apolar} to $F$, namely, it is contained in $\Ann^\circ(F) = \{G \in \R \ : \ G \circ F = 0\}$.
However, unlike $\Ann^\circ(F)$, it is also always (projectively) zero-dimensional and saturated, i.e., it defines a \emph{locally Gorenstein} punctual scheme of $\P^n(\k)$ \cite[§6.4]{IK99}.
Conversely, every locally Gorenstein punctual scheme is evinced by the GAD of some form, but this is not a correspondence as different GADs may evince the same scheme.

A scheme $\Z \subset \P^n(\k)$ of minimal length among all the schemes apolar to $F$ is simply referred to as \emph{minimal}, and its length is known as \emph{cactus rank} \cite{BB14} or \emph{scheme length} \cite{IK99} of $F$.
We will refer to the \emph{size} of a GAD as the length of the scheme that it naturally evinces. 

\begin{remark} 
    Since the natural scheme evinced by a GAD of $F\in \S_d$ is apolar to $F$, the cactus rank of $F$ is at most the length of that GAD.
    However, this inequality can be strict, as a minimal apolar scheme may not be evinced by a GAD of $F$ (for more details, see \Cref{rmk:localcactusvsGAD}). 
\end{remark}
\begin{example} \label{ex:ex1-A}
    Let $F = x_0^2x_1 + x_0x_1x_2 + x_1^3 \in \S_3$, and consider its local GAD $F = (x_0^2+x_0x_2+x_1^2)x_1$.
    Its divided power representation is $F_{\rm dp} = 2x_0^2x_1 + x_0x_1x_2 + 6x_1^3$, and its de-homogenization by $x_1=1$ (while fixing the other coordinates) is
    \[ f_{x_1} = 2x_0^2 + x_0x_2 + 6 \in \k[x_0,x_2]. \]
    Its annihilator by contraction is the ideal
    \[ \Ann^{\contract}(f_{x_1}) = \ideal{ (y_0-y_2)^2, y_2^2 } \subset \R', \]
    which is already homogeneous, hence it defines the homogeneous ideal $I(\Z) \subset \R$ associated with the considered GAD.
    The corresponding scheme $\Z$ is supported at $[0:1:0]$, which corresponds to $\ell = x_1$, and its Hilbert function is $H_{\Z} = (1,3,4,4,\dots)$.
    There are no apolar schemes to $F$ of length $3$, as $\Ann^{\circ}(F)$ is generated by 3 conics and it is zero-dimensional, so if $\Z' \subset \P^2(\k)$ had length 3 with $I(\Z') \subseteq \Ann^{\circ}(F)$, then $H_{\Z'}=(1,3,3,\dots)$ and therefore $I(\Z')$ would contain the 3 conics generating $\Ann^{\circ}(F)$.
    Since these conics define an Artinian ideal, $\Z'$ could not contain any projective point.
    Thus, $\Z$ is a minimal apolar scheme and the cactus rank of $F$ is $4$.
    In particular, there are neither local nor nonlocal GADs of $F$ of size $\leq 3$. 
\end{example}

In characteristic zero, the punctual schemes associated with GADs can also be defined as $\star$-annihilators of a prescribed linear functional, which is constructed from the given GAD \cite{BMT25}: for every term $\omega \ell^{d-k}$, we uniquely represent $\omega$ in $\k[\ell, x_1, \dots, x_n]$, namely 
\[ \omega = \sum_{j=0}^{k} w_{j} \ell^{k-j}, \quad \textnormal{where} \quad w_{j} \in \S'_{j}, \]
and we define 
\begin{equation} \label{eq:omegaell}
    \omega^{d,\ell} = \sum_{j=0}^{k} (d-j)! w_{j} \in \S'_{\leq k}.
\end{equation}
The linear functionals associated with a GAD, as in \cref{eq:GAD}, are then defined as
\begin{equation}\label{eq: exponential}
    \underline{\varphi} = \sum_{i=1}^s \omega_i^{d,\ell_i}\, \underline{\mathbf{e}}^{\ell_i} \in \k[[x_0,\dots,x_n]] \quad \textnormal{(homogeneous)}, \qquad \varphi = \sum_{i=1}^s \omega_i^{d,\ell_i}\, \mathbf{e}^{\ell_i} \in \k[[x_1,\dots,x_n]] \quad \textnormal{(affine)},
\end{equation} 
where $\underline{\mathbf{e}}^{\ell}=\sum_{m=0}^{\infty}\frac{\ell^m}{m!}$ is the usual exponential series and, for $\ell=x_0+\xi_1x_1+\dots+\xi_nx_n$, $\mathbf{e}^{\ell}=\left.\underline{\mathbf{e}}^{\ell}\right|_{x_0=0}=\exp(\xi_1x_1+\dots+\xi_nx_n)$.
The series $\underline{\varphi}$ (resp. $\varphi$) may be seen as a functional on $\R$ (resp. $\R'$) by identifying $\k[[x]] \simeq \R^*$, where $x^{\a}$ represents the dual element of $\frac{y^{\a}}{\a!}$.
It is easy to check that, under this identification, $\mathbf{e}^{\ell}$ represents the evaluation-in-$\xi$ functional, where $\xi = (\xi_1,\dots,\xi_n) \in \k^n$ is the point associated with $\ell$ \cite[§2.1.1]{M18}.
Then, it follows from \cite{BMT25} that the ideal of the scheme evinced by a GAD as in \cref{eq:GAD} can also be \emph{globally} computed as $\Ann^{\star}(\underline{\varphi}) = ( g \in \R_k \ : \ g \star \underline{\varphi} = 0 )_{k \in \N}$, while its affine projection is $\Ann^{\star}(\varphi) = \{ g \in \R' \ | \ g \star \varphi = 0\}$.

\begin{example} \label{ex:ex1-A-bis}
    Let $F = \omega\, x_1 \in \S_3$, with $\omega=x_0^2+x_0x_2+x_1^2$ as in \Cref{ex:ex1-A}.
    Following the construction discussed above, we write $x_0^2+x_0x_2+x_1^2 = x_1^2 + 0 x_1 + (x_0^2+x_0x_2)$, so we define
    \[ \omega^{3,x_1} = 3! + 1!(x_0^2+x_0x_2) = x_0^2 + x_0x_2 + 6 \in \k[x_0,x_2]. \]
    One can directly verify that
    \[ \Ann^{\star}\big(\omega^{3,x_1} \underline{\mathbf{e}}^{x_1} \big) = \ideal{ (y_0-y_2)^2,y_2^2 }. \]
    In fact, for every homogeneous $p \in \R$, we have
    \[ \big( (x_0^2 + x_0x_2 + 6) ( pq ) \big) (0,1,0) = 0,  \ \ \forall \, q \in \R \quad \iff \quad p \in \ideal{ (y_0-y_2)^2,y_2^2 }. \]
    For elements such as $p = 3y_0^2-y_1^2 \notin \ideal{ (y_0-y_2)^2,y_2^2 }$, the above expression may vanish for some value of $q \in \R$, e.g., for $q \in \R_{\leq 1}$, but not for all of them: for instance for $q = y_0^2$, we have $\big( (x_0^2 + x_0 x_2 + 6) (pq) \big)(0,1,0) = -2$.
\end{example}

In the following, we verify that the above constructions indeed define the same ideal, and that both $\omega^{d,\ell}$ and $f_{\ell}$ may fairly be considered \emph{local dual generators} for the resulting scheme.
For clarity, we consider a \emph{local GAD} ($s=1$ in \cref{eq:GAD}) with support $\ell = x_0$, and we prove the equality in the first affine chart.
The same argument holds for any choice of $\ell$ and any local chart.

\begin{proposition} \label{prop:equivgenerators}
    Let $\omega \in \S_k$ and $F = \omega x_0^{d-k} \in \S_d$. With the above notation, we have
    \[ (\omega^{d,x_0})_{\rm dp} = f_{x_0} \qquad \textnormal{and} \qquad  \Ann^{\star}( \omega^{d,x_0} \mathbf{e}^{x_0} ) = \Ann^{\contract}( f_{x_0} ). \]
\end{proposition}
\begin{proof}
    The first equality is a straightforward computation: let 
    \[ \omega = \sum_{\a_0=0}^k\sum_{|\a| = k - \a_0} w_{\a} x_0^{\a_0} x^{\a} \in \S_k, \quad \textnormal{so} \quad F = \sum_{\a_0=0}^k\sum_{|\a| = k - \a_0} w_{\a} x_0^{d-k+\a_0} x^{\a}. \]
    By definition (see \cref{eq:omegaell}), we then have
    \[ \omega^{d,x_0} = \sum_{\a_0=0}^k \sum_{|\a| = k - \a_0} (d-|\a|)! w_{\a} x^{\a}, \]
    while
    \[ F_{\rm dp} = \sum_{\a_0=0}^k\sum_{|\a| = k - \a_0} \a!(d-k+\a_0)! w_{\a} x_0^{d-k+\a_0} x^{\a}, \]
    therefore we conclude that
    \[ f_{x_0} = \sum_{\a_0=0}^k\sum_{|\a| = k - \a_0} \a!(d-|\a|)! w_{\a} x^{\a} = (\omega^{d,x_0})_{\rm dp}. \]

    For the second equality, we note that $p \in \Ann^{\star}( \omega^{d,x_0} \mathbf{e}^{x_0} )$ means
    \begin{equation} \label{eq:omegax}
        \omega^{d,x_0}(\partial)(pq)(0, \dots,0) = 0 \quad \forall \, q \in \R'.
    \end{equation} 
    For every $p,q \in \R'$, we have $\omega^{d,x_0}(\partial)(pq)(0) = ( q \contract (p \contract f_{x_0}) )(0)$.
    Taking $q = y^{\beta}$ recovers the coefficient of $x^{\beta}$ in $p \contract f_{x_0}$.
    Hence the left-hand side of \Cref{eq:omegax} vanishes for all $q$ if and only if $p \contract f_{x_0} = 0$.
\end{proof}

\begin{remark} \label{rmk:diminvsys}
    By \cref{eq:circvscontract}, the divided power operator ${\rm dp} : \S \to \S_{\rm dp}$ commutes with the $\R$-module structure on $\S$, and the same holds for $\R', \S'$.
    Thus, \Cref{prop:equivgenerators} shows that for every $F = \omega \ell^{d-k} \in \S_d$ there is a $\R'$-module isomorphism
    \[ \R' \circ \omega^{d,\ell} \simeq \R' \contract f_{\ell}. \]
    Hence, the above $\k$-vector spaces are isomorphic, and the isomorphism respects the polynomial degree.
\end{remark}

\subsection{Local GADs and inverse systems} As defined in \Cref{sec:GADsplus}, the size of a local GAD is the length of the naturally associated apolar scheme. We are particularly interested in \emph{minimal} ones.

\begin{definition}
    Let $F \in \S_d$.
    We define the \emph{local GAD-rank} of $F$ as the minimal size of a local GAD of $F$.
\end{definition}

\begin{remark} \label{rmk:localcactusvsGAD}
    The minimal length of a local scheme apolar to a given $F\in \S_d$, called the \emph{local cactus rank} of $F$ \cite{BJMR18}, may be strictly lower than its local GAD-rank, as minimal schemes do not need to be evinced by a GAD of $F$, even if they are local.
    Indeed, as shown in \cite[Prop. 1]{BJMR18} and \cite[Remark~4.2]{BOT24}, if $Z$ is an \emph{irredundant} local apolar scheme, 
    then $Z$ is locally defined by $\Ann^{\contract}(g)$ for some $g \in \S'$ whose degree-$d$ tail equals $f_\ell$.
    Thus, the local cactus rank of $F$ may be evinced by a GAD of a polynomial $G\in \S_{d'}$ with $d'>d$, of which $F$ is a partial (see \cite[Example 4.6]{BOT24} for an explicit example). 
    More generally, the (global) GAD-rank of $F$, that is, the minimal length of a (not necessarily local) scheme evinced by a GAD of $F$, is the minimal length of a zero-dimensional scheme apolar to $F$ and contained in a union of $(d + 1)$-fat points \cite[Prop. 4.3]{BOT24}. 
\end{remark}

We can leverage Macaulay's correspondence \cite{M19} to compute the length of the scheme $\Z$ evinced by a local GAD $F = \omega \ell^{d-k}$ as the $\k$-dimension of $\R' \contract f_{\ell}$ (equiv., $\R' \circ \omega^{d,\ell}$, by \Cref{rmk:diminvsys}).
This is a computationally friendly method, as it involves only linear algebra operations and does not even require computing the defining ideal $I(\Z)$.
We state this precisely in the notation of this paper, although we remark that several formulations of the following result, in different languages, have already appeared in the literature (see, e.g., \cite{BJMR18,CI12,M18}).

\begin{lemma} \label{lemma:schemedim}
    Let $\Z$ be the natural apolar scheme of $F = \omega \ell^{d-k} \in \S_d$ at $\ell \in \S_1$. Then
    \[ \len \Z = \dim_{\k} (\R' \contract f_{\ell}). \]
\end{lemma}
\begin{proof}
    We have a natural $\R'$-module isomorphism
    \[ \R'/\Ann^{\contract}(f_{\ell}) \to \R' \contract f_{\ell}, \quad g \mapsto g \contract f_\ell, \]
    hence the above $\k$-vector spaces have the same dimension.
    Since $I(\Z)$ is obtained by homogenizing $\Ann^{\contract}(f_{\ell})$ by $\ell$, for every $t\geq 0$ the degree-$t$ homogenization map is a $\k$-vector space isomorphism \cite[Chap. 9, Thm. 3.12]{CLOS07}
    \[ \big(\R/I(\Z)\big)_t \simeq \R'_{\leq t}/\Ann^{\contract}(f_{\ell})_{\leq t}. \] 
    For large enough $t \in \N$, the dimension of the left hand side agrees with $\len \Z$, while the dimension of the right one yields $\dim_{\k} \R'/\Ann^{\contract}(f_{\ell}) = \dim_{\k} (\R' \contract f_{\ell})$.
\end{proof}

\begin{example}
    In \Cref{ex:ex1-A}, we computed the length-$4$ scheme evinced by the local GAD $F = (x_0^2+x_0x_2+x_1^2)x_1 \in \S_3$, and noted that this is the minimal length of such an apolar scheme.
    Given \Cref{lemma:schemedim}, we can compute the same measure by
    \[ \dim_{\k} (\R' \contract f_{x_1}) = \dim_{\k} \langle 2x_0^2 + x_0x_2 + 6, 2x_0 + x_2, x_0, 1 \rangle = 4. \]
    By \Cref{rmk:diminvsys}, this is also equal to
    \[ \dim_{\k} (\R' \circ \omega^{3,x_1}) = \dim_{\k} \langle x_0^2 + x_0x_2 + 6, 2x_0 + x_2, x_0, 1 \rangle = 4. \]
\end{example}

\begin{definition}
     Let $F \in \S_d$ and $\ell \in \S_1$. We define the \emph{inverse system of $F$ at $\ell$} as the $\k$-vector space
     \[ \R' \contract f_{\ell} = \{g \contract f_{\ell} \ : \ g \in \R' \}. \]
\end{definition}
We note that the name is consistent with what is usually referred to as the inverse system of an ideal $I \subseteq \R'$ with respect to contraction \cite{CI12}, namely
\[ I^{-1} = \{f \in \S' \ : \ I \contract f = 0 \}. \]
Indeed, we have $\R'\contract f_{\ell} = \big(\Ann^{\contract}(f_{\ell})\big)^{-1}$, since $g\in \big(\Ann^{\contract}(f_{\ell})\big)^{-1}$ if and only if $\Ann^{\contract}(f_{\ell})\subseteq \Ann^{\contract}(g)$, which is equivalent to $g$ being a contraction of $f_{\ell}$, i.e., $g \in \R'\contract f_{\ell}$.


\section{Minimal local GAD} \label{sec:minLocalGAD}

We now focus on the main objective of this paper, namely, computing minimal local GADs $F = \omega \ell^{d-k}$ for a given $F \in \S_d$.
This is equivalent to finding the \emph{minimal supports} $\ell \in \S_1$ such that the natural apolar scheme to $F$ at $\ell$, as defined in \Cref{sec:GADsplus}, has minimal length.
We immediately note that this equivalence holds only in the local setting, as the scheme is all concentrated at one point.
This constitutes a major obstruction for extending the methods developed in this paper to understand non-local objects.

To compute such minimal supports, we symbolically construct $\R' \contract f_{\ell}$, in the parameters of $\ell$, and minimize its dimension.
We work with $\R' \contract f_{\ell}$ instead of $\R' \circ \omega^{d,\ell}$ because the \emph{inverse system matrix} 
will be a \emph{Hankel matrix} (see \Cref{rmk:HankelStructure}).

\subsection{Symbolic inverse system} \label{subsec:symbinvsys}

Let $\K = \k[\c]$ be a symbolic polynomial ring in $\c = (\c_1, \dots, \c_n)$, and consider a symbolic support $\ell = x_0 + \c_1 x_1 + \dots + \c_n x_n \in \K[x]$.
To compute the dual generator $f_{\ell}$, we remove the dependence on the symbols in the linear form where we localize.
We (arbitrarily) consider the base change $x_0 \mapsfrom \ell$ and $x_i \mapsfrom x_i$ for $1 \leq i \leq n$, given by the unipotent upper triangular matrix
\[
    \phi_{\c} = I_{n+1} - (1,0,\dots,0)^T(0,\c_1,\dots,\c_n) = \begin{pmatrix}
1 & -\c_1 & \dots & -\c_n\\
0 & 1     & \dots & 0 \\
\vdots & \vdots & & \vdots \\
0 & 0     & \dots & 1 
\end{pmatrix}.
\]
The inverse systems of $F$ at $\ell$ and of $F\big(\phi_{\c}(x) \big)$ at $x_0 = \phi_{\c}(\ell)$ are isomorphic through $\phi_{\c}$ \cite[Cor. 3.8]{BMT25}.
In particular, they have the same dimension for each specialization of parameters $\c$ in $\k^n$, hence we can equivalently aim to minimize the latter.
To this end, we define the \emph{symbolic dual generator} as
\begin{equation} \label{eq:Fell}
    f_{\c} = F\big(\phi_{\c}(x) \big)_{\rm dp}(1,x_1, \dots,x_n) \in \K[x_1,\dots,x_n]_{\leq d}.
\end{equation}

\begin{definition} \label{defn:IFell}
    Let $F \in \S_d$, $\ell = x_0 + \c_1 x_1 + \dots + \c_n x_n \in \K[x]$ and $f_{\c}$ be defined as in \cref{eq:Fell}.
    We define the \emph{inverse system matrix} of $F$ at $\ell$ as the square matrix $\mathcal{I}_{F, \ell}(\gamma)$, indexed by $\{y^{\a}\}_{|\a| \leq d}$, given by 
\[
\mathcal{I}_{F,\ell}(\gamma) = \begin{array}{c@{\;}c}
 & y^{\b} \\[-2pt]
y^{\a} &
\left( \begin{array}{ccc}
 & \vdots & \\
\dots & \big(y^{\alpha+\beta} \contract f_{\c}\big)(0,\dots,0) & \dots \\
 & \vdots &
\end{array} \right).
\end{array}
\]
\end{definition}
We label the rows and columns of the inverse system matrix in the degree lexicographic order (with $y_1<\ldots<y_n$), but we remark that the choices of the monomial order or the basis of $\R'_{\leq d}$ do not affect the method we present.

\begin{remark} \label{rmk:HankelStructure}
    The inverse system matrix is a Hankel matrix, as its entries only depend on the product of the row and column indices, i.e., on the sum of the corresponding exponent vectors.
    Indeed, the row indexed by $y^{\a}$ lists the coefficients of $y^{\a} \contract f_{\c}$ in the monomial basis $\{x^{\b}\}_{|\b|\leq d}$ of $\S'_{\leq d}$.
    Moreover, by construction, the symbolic monomial coefficient of $x^{\b}$ in $f_{\c}$ is a polynomial in $\K$, whose degree in each $\c_i$ is not greater than the corresponding $\b_i$.
    In particular, the coefficient of every pure power monomial $x_i^{e_i}$ in $f_{\c}$ is a univariate polynomial in $\k[\c_i]_{\leq e_i}$, and the $(y^{\a},y^{\b})$-entry of $\mathcal{I}_{F,\ell}(\gamma)$ lies in $\k[\c]_{\leq |\a|+|\b|}$.
\end{remark}

\begin{example} \label{ex:ex1-B}
    Let $F=x_0^2x_1 + x_0x_1x_2 + x_1^3 \in \S_3$ be the same form of \cref{ex:ex1-A}, and consider a symbolic linear form $\ell = x_0 + \gamma_1 x_1 + \gamma_2 x_2 \in \K[x_0,x_1,x_2]$.
    We compute $f_{\c}$ directly from \cref{eq:Fell} as
    \[ f_{\c} = (6\gamma_1^2 + 6) x_1^3 + (4\gamma_1\gamma_2 - 2\gamma_1) x_1^2x_2 + (2\gamma_2^2 - 2\gamma_2) x_1x_2^2 - 4\gamma_1 x_1^2 + (-2\gamma_2 + 1) x_1x_2 + 2x_1. \]
    We construct the inverse system matrix $\mathcal{I}_{F,\ell}(\gamma)$ as in \Cref{defn:IFell}, by listing the representation of every monomial contraction of $f_{\c}$ in the monomial basis of $\K[x_1,x_2]$:
\[
\hspace{-0.3cm} \scalebox{0.88}{$ \begin{array}{c@{\;}c}
&
\begin{array}{cccccccccc}
\hspace{1.9cm} 1 \hspace{1.55cm} & y_1\hspace{1.45cm} & y_2\hspace{1.5cm} & y_1^2\hspace{1.25cm} & y_1y_2\hspace{1.2cm} & y_2^2\hspace{0.95cm} & y_1^3\hspace{0.95cm} & y_1^2y_2\hspace{0.95cm} & y_1y_2^2 \hspace{0.35cm} & y_2^3 \hspace{1.15cm}
\end{array}
\\[6pt]
\begin{array}{c}
1 \contract f_{\c}\\
y_1 \contract f_{\c}\\
y_2 \contract f_{\c}\\
y_1^2 \contract f_{\c}\\
y_1y_2 \contract f_{\c}\\
y_2^2 \contract f_{\c}\\
y_1^3 \contract f_{\c}\\
y_1^2y_2 \contract f_{\c}\\
y_1y_2^2 \contract f_{\c}\\
y_2^3 \contract f_{\c}
\end{array}
&
\left(
\begin{array}{cccccccccc}
0 & 2 & 0 & -4\gamma_1 & -2\gamma_2+1 & 0 & 6\gamma_1^2+6 & 4\gamma_1\gamma_2-2\gamma_1 & 2\gamma_2^2-2\gamma_2 & 0 \\
2 & -4\gamma_1 & -2\gamma_2+1 & 6\gamma_1^2+6 & 4\gamma_1\gamma_2-2\gamma_1 & 2\gamma_2^2-2\gamma_2 & 0 & 0 & 0 & 0 \\
0 & -2\gamma_2+1 & 0 & 4\gamma_1\gamma_2-2\gamma_1 & 2\gamma_2^2-2\gamma_2 & 0 & 0 & 0 & 0 & 0 \\
-4\gamma_1 & 6\gamma_1^2+6 & 4\gamma_1\gamma_2-2\gamma_1 & 0 & 0 & 0 & 0 & 0 & 0 & 0 \\
-2\gamma_2+1 & 4\gamma_1\gamma_2-2\gamma_1 & 2\gamma_2^2-2\gamma_2 & 0 & 0 & 0 & 0 & 0 & 0 & 0 \\
0 & 2\gamma_2^2-2\gamma_2 & 0 & 0 & 0 & 0 & 0 & 0 & 0 & 0 \\
6\gamma_1^2+6 & 0 & 0 & 0 & 0 & 0 & 0 & 0 & 0 & 0 \\
4\gamma_1\gamma_2-2\gamma_1 & 0 & 0 & 0 & 0 & 0 & 0 & 0 & 0 & 0 \\
2\gamma_2^2-2\gamma_2 & 0 & 0 & 0 & 0 & 0 & 0 & 0 & 0 & 0 \\
0 & 0 & 0 & 0 & 0 & 0 & 0 & 0 & 0 & 0
\end{array}
\right).
\end{array} $}
\]
\end{example}

We now turn the observations of \Cref{rmk:HankelStructure} into an effective criterion to ensure a finite number of minimal supports.

\begin{lemma} \label{lem:degfL}
    The size of the local GAD $F = \omega\ell^{d-k} \in \S_d$ is at least $1+\deg f_{\ell}$.
\end{lemma}
\begin{proof}
    Let $x^{\a}$ be a monomial of maximal degree in $f_{\ell}$ and consider, for $1 \leq N \leq |\a| = \deg f_{\ell}$, a chain of contractions
    \[ Y_0 = 1 \in \R', \quad Y_{N} = Y_{N-1}\,y_i, \ \textnormal{ where } y_i \in \R' \textnormal{ is chosen such that } \ Y_{N-1}\,y_i \ | \ y^{\a}. \]
    Since $\deg (Y_i \contract f_{\ell}) = |\a|-i$, the contractions $\{Y_i \contract f_{\ell} \}_{0 \leq i \leq |\a|}$ are $\k$-linearly independent elements of $\R' \contract f_{\ell}$, so by \Cref{lemma:schemedim} the local GAD-rank of $F$ at $\ell$ is at least $|\a|+1$.
\end{proof}

\begin{proposition} \label{prop:FinitelyManyMinimalSupports}
    Let $F \in \S_d$ be a form of local GAD-rank at most $d$.
    Then $F$ has at most $d^n$ minimal supports.
\end{proposition}
\begin{proof}
    After a generic change of coordinates, we can assume that $F = \sum_{\a_0 = 0}^d \sum_{|\a|=d-\a_0} F_{\a} x_0^{\a_0} x^{\a}$, with $F_{0} \neq 0$, i.e., the coefficient of $x_0^{d}$ in $F$ is non-zero.
    Thus, we have
    
    \begin{equation} \label{eq:fgammaexps}
    f_{\c} = \sum_{|\a|\leq d} F_{\a} \left.\Big((x_0-\c_1x_1-\dots-\c_nx_n)^{d-|\a|}x^{\a}\Big)_{\rm dp}
    \right|_{x_0=1}.
    \end{equation}
    Let $\ell \in \S_1$ be a minimal support of $F$.
    If $f_{\ell}$ had degree $d$, then $F$ would have local GAD-rank $> d$ by \Cref{lem:degfL}. 
    Therefore, we conclude that all the coefficients of degree-$d$ monomials of $f_{\ell}$ must vanish for such a minimal $\ell$.
    By inspecting \cref{eq:fgammaexps}, we see that the monomial coefficient of $x_i^d$ is of the form $d!F_{0}(-\c_i)^d + \k[\c_i]_{\leq d-1}$, namely, a degree-$d$ univariate polynomial in $\c_i$.
    Hence, there are at most $d$ choices for each $\c_i$, so at most $d^n$ choices for such a minimal $\ell$.
    Since the above argument was performed on a generic chart of $\P^n(\k)$, the minimal supports are all computed as above.
\end{proof}

\begin{remark}
    The proof of \Cref{prop:FinitelyManyMinimalSupports} shows that the coefficients of degree-$d$ pure powers in $f_{\c}$ define a zero-dimensional ideal in $\k[\c]$.
    However, it also shows that \emph{every} coefficient of degree-$d$ terms in $f_{\c}$ must vanish for a minimal support, if the corresponding rank does not exceed $d$.
    Thus, except for small cases (e.g., $F = x_0x_1 \in \S_2$ has $2$ minimal supports), the bound of \Cref{prop:FinitelyManyMinimalSupports} is not expected to be sharp: the additional vanishing conditions typically imply that the number of minimal supports is much lower than $d^n$.
\end{remark}

\begin{corollary} \label{cor:localcactus}
    Let $F \in \S_d$ be a form of local cactus rank at most $d$.
    Then $F$ has at most $d^n$ minimal local apolar schemes.
\end{corollary}
\begin{proof}
    Let $\Z \subset \P^n(\k)$ be a minimal local apolar scheme to $F$.
    By \cite[Prop. 2.11-(iv)]{CI12}, the socle degree $\mathfrak{a}(\Z)$ of the Artinian algebra defining $\Z$ satisfies $\mathfrak{a}(\Z) + 1 \leq \len \Z \leq d$, hence by \cite[Prop. 4.3]{BOT24} the scheme $\Z$ is evinced by a (minimal) local GAD of $F$.
    Therefore, the local GAD-rank of $F$ agrees with its local cactus rank, and the minimal local apolar schemes to $F$ arise from minimal local GADs, which are at most $d^n$ by \Cref{prop:FinitelyManyMinimalSupports}.
\end{proof}

\subsection{Minimal local GAD algorithm}

As observed in \Cref{subsec:symbinvsys}, a minimal local GAD of a form $F\in \S_d$ is determined by a linear form $\ell\in \S_1$ minimizing the dimension of the inverse system $\R' \contract f_{\ell}$.
Thus, the coefficients of such $\ell$ minimize the rank of the inverse system matrix $\mathcal{I}_{F, \ell}(\gamma)$.

\vspace{-0.1cm}
\noindent\begin{minipage}{\linewidth}
\begin{algorithm}[H]
\centering
\caption{\texttt{Minimal local GAD}\label{algo:localgad}}
\begin{algorithmic}[1]
\Function{MinimalSupports}{$F$}
    \State \textbf{Input}: $F \in \S_d$.
    \State Compute $f_{\c}$ for a symbolic linear form $\ell = x_0 + \c_1 x_1 + \dots + \c_nx_n$.
    \State Construct the inverse system matrix $\mathcal{I}_{F,\ell}(\gamma)$.
    \State\label{algo-hard}Compute the values of $\c$ minimizing the rank of $\mathcal{I}_{F,\ell}(\gamma)$.
    \State \textbf{Output}: The coefficients $\c$ of all minimal supports of $F$ in the first affine chart.
\EndFunction
\end{algorithmic}
\end{algorithm}
\end{minipage}

\vspace{0.1cm}
\begin{remark} \label{rmk:difficultPart}
    The computational bottleneck of \Cref{algo:localgad} is Line \ref{algo-hard}. 
    Theoretically, we could consider the ideal generated by all the $r \times r$ minors of $\mathcal{I}_{F,\ell}(\gamma)$, for increasing values of $r \in \N$, until the minors have a common zero in $\k^n$.
    If the $r \times r$ minors define a zero-dimensional locus (over $\overline{\k}$), we may collect such minors until they generate a zero-dimensional ideal and then retain those of its $\k$-rational zeros at which the matrix has rank $r-1$.
    We can also exploit the Hankel structure to avoid redundant minors.  
    A more detailed discussion of this method is given in \Cref{sec:RankMinimiz}.
\end{remark}

\begin{example} \label{ex:ex1-C}
    We consider $F = x_0^2x_1 + x_0x_1x_2 + x_1^3 \in \S_3$ as in the previous examples.
    Since its cactus rank is $4$, as observed in \cref{ex:ex1-A}, there will always be a non-vanishing $4 \times 4$ minor of $\mathcal{I}_{F,\ell}(\gamma)$, regardless of the choice of the parameters $\{\gamma_1,\gamma_2\}$.
    This can also be checked computationally, by verifying that the ideal defined by the $4 \times 4$ minors of $\mathcal{I}_{F,\ell}(\gamma)$ is the unit ideal. 
    To search for minimal supports of length-$4$ local GADs, we consider the following $5\times5$ submatrix of $\mathcal{I}_{F,\ell}(\gamma)$:
\[
\begin{array}{c@{\;}c}
&
\begin{array}{ccccc}
\hspace{0.5cm} y_1 \hspace{1.25cm} & y_2 \hspace{1.15cm} & y_1y_2 \hspace{1.15cm} & y_2^2 \hspace{0.95cm} & y_1y_2^2 \hspace{0.35cm}
\end{array}
\\[6pt]
\begin{array}{c}
1\contract f_{\c}\\
y_1\contract f_{\c}\\
y_2\contract f_{\c}\\
y_1y_2\contract f_{\c}\\
y_2^2\contract f_{\c}
\end{array}
&
\left(
\begin{array}{ccccc}
2 & 0 & -2\gamma_2+1 & 0 & 2\gamma_2^2-2\gamma_2 \\
-4\gamma_1 & -2\gamma_2+1 & 4\gamma_1\gamma_2-2\gamma_1 & 2\gamma_2^2-2\gamma_2 & 0 \\
-2\gamma_2+1 & 0 & 2\gamma_2^2-2\gamma_2 & 0 & 0 \\
4\gamma_1\gamma_2-2\gamma_1 & 2\gamma_2^2-2\gamma_2 & 0 & 0 & 0 \\
2\gamma_2^2-2\gamma_2 & 0 & 0 & 0 & 0
\end{array}
\right).
\end{array}
\]
Its determinant $2^5\gamma_2^5(\gamma_2-1)^5$ vanishes if and only if $\gamma_2 \in \{0,1\}$.
We now consider another $5 \times 5$ submatrix:
\[
\begin{array}{c@{\;}c}
&
\begin{array}{ccccc}
\hspace{0.25cm} y_1 \hspace{1.45cm} & y_2 \hspace{1.45cm} & y_1^2 \hspace{1.3cm} & y_1y_2 \hspace{1.15cm} & y_1^2y_2 \hspace{0.2cm}
\end{array}
\\[6pt]
\begin{array}{c}
1\contract f_{\c}\\
y_1\contract f_{\c}\\
y_2\contract f_{\c}\\
y_1^2\contract f_{\c}\\
y_1y_2\contract f_{\c}
\end{array}
&
\left(
\begin{array}{ccccc}
2 & 0 & -4\gamma_1 & -2\gamma_2+1 & 4\gamma_1\gamma_2-2\gamma_1 \\
-4\gamma_1 & -2\gamma_2+1 & 6\gamma_1^2+6 & 4\gamma_1\gamma_2-2\gamma_1 & 0 \\
-2\gamma_2+1 & 0 & 4\gamma_1\gamma_2-2\gamma_1 & 2\gamma_2^2-2\gamma_2 & 0 \\
6\gamma_1^2+6 & 4\gamma_1\gamma_2-2\gamma_1 & 0 & 0 & 0 \\
4\gamma_1\gamma_2-2\gamma_1 & 2\gamma_2^2-2\gamma_2 & 0 & 0 & 0
\end{array}
\right).
\end{array}
\]
When $\gamma_2 \in \{0,1\}$, the determinant of this second submatrix is $\pm 32\gamma_1^5$, which vanishes precisely when $\gamma_1=0$.
Thus, we have only two possible candidates to annihilate all $5 \times 5$ minors, namely $(\gamma_1,\gamma_2) = (0,0)$ and $(\gamma_1,\gamma_2) = (0,1)$.
One can easily check that both choices lead to rank-$4$ numeric matrices, hence we get two minimal local GADs
\[ F = F\, x_0^0 \quad \textnormal{and} \quad F = F\, (x_0+x_2)^0, \]
whose natural apolar schemes are defined, respectively, by
\[ \ideal{-6y_0y_2 + y_1^2, y_2^2 } \quad \textnormal{and} \quad \ideal{ -6y_0(y_0-y_2) + y_1^2, (y_0-y_2)^2 }. \]
Their associated Hilbert function is $(1,3,4,4,\dots)$, equal to that of the minimal scheme computed in \cref{ex:ex1-A}.
\end{example}

\begin{remark} \label{rmk:threerank4}
    In \cref{ex:ex1-C} we did \emph{not} find the minimal local GAD $F = (x_0^2+x_0x_2+x_1^2)x_1$ that we considered in \cref{ex:ex1-A}.
    This was expected, since the support $[0:1:0]$ of the latter GAD lives in the projective hyperplane $x_0 = 0$.
    By running \Cref{algo:localgad} after a generic change of coordinates, we verify that these three local GADs are the unique minimal ones.
    We also note that the minimal apolar schemes of \cref{ex:ex1-C} are $\PGL_3(\k)$-equivalent, but their $\PGL_3(\k)$-orbit does not contain the scheme of \cref{ex:ex1-A}, whose degree-$2$ graded component is made of forms in two \emph{essential variables} \cite{C05}.
\end{remark}

\begin{remark} \label{rmk:norank5}
    Computing the ideal generated by $6\times 6$ minors of $\mathcal{I}_{F,\ell}(\gamma)$ in \Cref{ex:ex1-C}, we find again the same three solutions that were defined by the $5\times 5$ minors.
    This holds even after a random linear change of coordinates on $F$. 
    Hence, we conclude that there are no local GADs of $F$ of size $5$.
    We note that this holds only locally, as the GAD $F = x_0(x_0x_1+x_1x_2) + x_1^3$ evinces the scheme $\Z$ defined by
    \[ I(\Z) = \ideal{ y_0y_1^2, y_1^2y_2, y_2^2 } = \ideal{y_1^2,y_2^2} \cap \ideal{y_0,y_2}, \]
    whose Hilbert function is $H_{\Z} = (1,3,5,5,\dots)$.
    We also note that we cannot have such \emph{length jumps} with unconstrained apolar schemes, since for every length-$r$ scheme apolar to $F$, we can obtain infinitely many length-$(r+j)$ apolar schemes to $F$, by simply adding $j$ new reduced points.
\end{remark}

\subsection{Bounds on the local GAD-rank} \label{sec:boundslocalGADrank}
In this section, we discuss the constraints on the local GAD-rank.

By \Cref{lem:degfL}, the local GAD-rank of a form $F \in \S_d$ is bounded from below by $d-k+1$, where $k$ is the maximal power of a linear form dividing $F$.
Another lower bound is given by the number of essential variables of $F$, that is, the minimum number of variables required to write the polynomial, up to $\GL$-action \cite{C05}.
This number equals the dimension of the space spanned by the linear derivation of $F$ (equivalently, the linear contractions of $F_{\operatorname{dp}}$, or the rank of the first \emph{catalecticant matrix} of $F$, see \Cref{subsec:catalecticant}), and it is a lower bound for the length of any apolar scheme to $F$.
Indeed, if $F$ has $n+1$ essential variables and $\Z \subset \mathbb P^n(\k)$ is an apolar scheme to $F$ of length $r<n+1$, then $I(\Z)_1$ contains at least one linear form, so $I(\Z) \subseteq \Ann^{\circ}(F)$ implies that some variable of $F$ was not essential.
This lower bound can be attained, for instance when $F = g \ell^{d-1}$ for $g,\ell \in \S_1$ and $d \geq 2$.
In such a case, the number of essential variables and the local GAD-rank are both equal to $\dim_{\k}\langle g,\ell\rangle$.
Both of the previous quantities are at most the local cactus rank of $F$, which is therefore a theoretically better lower bound for the local GAD-rank of $F$ (cf. \Cref{rmk:localcactusvsGAD}).
However, we currently have no general methods to efficiently compute this rank.

The next proposition, based on the ideas in {\cite[Thm. 3]{BR13}}, provides an upper bound on the length of local GADs.

\begin{proposition} \label{prop:genericrank}
    Let $F \in \S_d$ and let $\Z$ be a scheme evinced by a local GAD of $F$. Then
    \[ \len \Z \leq \begin{cases} \frac{n+2k}{k}\binom{ n+k-1 }{n} & \textnormal{if } d = 2k, \\
    2\binom{ n+k }{n} & \textnormal{if } d = 2k+1.
    \end{cases}
    \]
\end{proposition}
\begin{proof}
    For any $\ell \in \S_1$, $f_\ell$ is a polynomial of degree at most $d$ in $n$ variables.
    Set $k = \lfloor d/2\rfloor$. There are $\binom{n+k}{n}$ monomials $y^{\b}$ with $|\b|\leq k$, while all contractions $y^{\b}\contract f_\ell$ with $|\b|>k$ lie in $\S'_{\leq d-k-1}$.
    Hence $\dim_{\k}(\R'\contract f_\ell)$ is at most $\binom{n+k}{n}+\binom{n+d-k-1}{n}$, which equals the bound above, as in the proof of \cite[Thm.~3]{BR13}.
    By \Cref{lemma:schemedim}, this number bounds $\len \Z = \dim_{\k}(\R' \contract f_\ell)$ from above.
\end{proof}

\begin{remark}
The formula in \Cref{prop:genericrank} corresponds to the length of the apolar algebra of a generic non-homogeneous polynomial (see \cite[Thm. 3.31]{IE78} and \cite[Thm. 1D]{Iar84}).
Since the de-homogenization by $\ell$ induces an isomorphism of $\k$-vector spaces between $\S_d$ and $\S'_{\leq d}$, this is the expected length of a local GAD $F\ell^0$, for generic choices of $F \in \S_d$ and $\ell \in \S_1$.
However, the local GAD-rank of a generic $F \in \S_d$ does not need to match this upper bound, as shown by the following example.
In fact, generic forms may admit non-generic de-homogenizations, for special choices of $\ell \in \S_1$.
\end{remark}

We thank the anonymous referee for pointing out the following example.

\begin{example} \label{ex:genericnotmax}
     A general element of $\k[x_0,x_1,x_2]_3$ defines an elliptic curve $V(F) \subset \P^2(\k)$.
     For a generic $\ell \in \S_1$, $f_\ell$ has 6 linearly independent contractions, which agrees with the bound given by \Cref{prop:genericrank}.
     On the other hand, let $\ell \in \S_1$ be the equation of an inflectional tangent of the curve, that is, a tangent to $V(F)$ at an inflection point.
     We note that there are precisely $9$ distinct choices for such lines (possibly in some finite extension of $\k$), since the inflection points arise from the intersection of the elliptic curve with its \emph{Hessian} (see, e.g., \cite{MPT25}). 
     We consider the coordinate change sending the inflectional line at infinity ($\ell = x_2$), and $F = - x_1^2x_2 + x_0^3 + a x_0 x_2^2 + b x_2^3$ in its short Weierstrass form, with $a,b \in \k$.
     Thus, we have $f_{\ell} = -2x_1^2 + 6x_0^3 + 2a x_0 + 6b$ and
     \[ \dim_{\k} ( \R' \contract f_\ell ) = \dim_{\k} \langle f_\ell, y_0 \contract f_\ell, y_0^2 \contract f_\ell, y_0^3 \contract f_\ell, y_1 \contract f_\ell \rangle = 5. \]
     Thus, for a generic ternary cubic $F$, the size of the local GAD drops from $6$ for a generic $\ell$ to $5$ when $\ell$ defines one of the $9$ inflectional tangents of $V(F)$.
\end{example}


\begin{remark}
    The upper bound given in \Cref{prop:genericrank} can be employed to describe the size of \emph{all} possible local GADs of a form (see, e.g., \Cref{ex:runningexample}).
    We note that one can find positive-dimensional loci of supports determining a given GAD size even before reaching the maximal one given by \Cref{prop:genericrank}.
    These loci may (\Cref{ex:InfiniteMinimalLocus}) or may not (\Cref{ex:funnymonomial}) constitute minimal supports of $F$.
\end{remark}

\begin{example} \label{ex:runningexample}
    We consider once more $F = x_0^2x_1 + x_0x_1x_2 + x_1^3 \in \S_3$.
    In \Cref{ex:ex1-A,ex:ex1-C} we proved that this form has no local GADs of size $\leq 3$, and three local GADs of minimal size $4$.
    We also observed in \Cref{rmk:norank5} that it has no local GADs of size $5$.
    Since $6$ is the maximal possible size for a local GAD by \Cref{prop:genericrank}, we conclude that every local GAD $F = F \ell^0$, for any $\ell \in \S_1 \setminus\{x_1,x_0,x_0+x_2\}$, evinces a scheme of length $6$.
\end{example}

\begin{example} \label{ex:InfiniteMinimalLocus}
    Let us consider $F = x_0x_1+x_0x_2+x_1x_2 \in \S_2$.
    We check that the symbolic rank of $\mathcal{I}_{F,\ell}(\c)$ is $4$, which is the maximum that can be achieved by \Cref{prop:genericrank}.
    Hence, a generic $\ell \in \S_1$ will evince a length-$4$ natural apolar scheme to $F$.
    Furthermore, we verify that the $3 \times 3$ minors of $\mathcal{I}_{F,\ell}(\c)$ define the unit ideal, so the only possible sizes for local GADs of $F$ are $3$ and $4$.
    A straightforward computation shows that the non-zero $4 \times 4$ minors are all multiples of 
    \[ g(\c) = \c_1^2 - 2\c_1\c_2 - 2\c_1 + \c_2^2 - 2\c_2 + 1. \]
    In particular, there is a $1$-dimensional locus of minimal supports for $F$.
    Solving $g(\c)=0$ for the indeterminate $\c_2$, we find that $\{ \ell = x_0 + \gamma^2 x_1 + (1 + \gamma)^2x_2 \}_{\gamma \in \k}$ are all minimal supports, evincing length-$3$ natural apolar schemes to $F$.
\end{example}

\begin{example} \label{ex:funnymonomial}
    We consider $F = x_0^2(x_0+x_1)(x_0+x_2) \in \S_4$.
    The symbolic rank of $\mathcal{I}_{F,\ell}(\c)$ is 9, which coincides with the bound given by \Cref{prop:genericrank}.
    By considering the minors of $\mathcal{I}_{F,\ell}(\c)$, we conclude that $F$ has a unique minimal local GAD supported at $x_0$, with size $4$, no local GADs of size $5$, two local GADs of size $6$, supported at $x_0+x_1$ and $x_0+x_2$, and no local GADs of size $7$.
    However, one can check that we have two $1$-dimensional families of supports, namely $\{ \ell=x_0+\c_1 x_1 \}_{\c_1 \neq 0,1}$ and $\{ \ell = x_0+\c_2 x_2 \}_{\c_2 \neq 0,1}$, evincing schemes with Hilbert function $(1,3,5,7,8,8, \dots)$, namely, their size is strictly smaller than the generic local GAD size.
\end{example}

\subsection{Rank minimization} \label{sec:RankMinimiz}

In this section, we discuss a computational realization of Line \ref{algo-hard} in \Cref{algo:localgad}.
As observed in \Cref{rmk:difficultPart}, computing all $r \times r$ minors of $\mathcal{I}_{F,\ell}(\gamma)$ is infeasible even in moderately sized cases.
We will address forms that admit a \emph{finite} number of minimal supports. 
To this end, we collect minors from $\mathcal{I}_{F,\ell}(\gamma)$ until they define a zero-dimensional variety (over $\overline{\k}$) in the parameter space $\k[\gamma]$.
The efficiency of such an approach heavily depends on how these minors are chosen in practice.
We tested the following three methods:
\begin{enumerate}[leftmargin=0.7cm]
    \item[(A)] Extracting random minors, obtained from a uniform choice of rows and columns.
    \item[(B)] Uniformly sampling pairs of rows and columns $(y^\a,y^\b)$ such that $|\a+\b|=\deg(f_{\ell})$.
    \item[(C)] Selecting random chains of consecutive contractions from a leading monomial $x^{\b}$ of $f_{\ell}$: starting from $(\a_0,\b_0)=(1,\b)$, define inductively 
    \[ \a_{m+1} = (\a_{m,0}, \dots, \a_{m,k}+1, \dots, \a_{m,n}), \quad \b_{m+1} = (\b_{m,0}, \dots, \b_{m,k}-1, \dots, \b_{m,n}), \]
    where $0 \leq k \leq n$ is chosen with probability $\P(k = j) = \frac{ \beta_{m,j} }{|\b_m|}$.
    If we collected less than $r$ rows and columns indices $(y^{\a_m},y^{\b_m})$ when $\b_m$ reaches the zero-vector, we start over with another monomial of $f_{\ell}$, and keep only indices that have not appeared before.
\end{enumerate}

A visual presentation of the three considered strategies is given in \Cref{fig:threeApproaches}.
The approach (A) is the naive minor selection, while (B) exploits the Hankel structure of $\mathcal{I}_{F,\ell}(\gamma)$ to create block-triangular submatrices.
This substantially reduces the choices of vanishing minors and simplifies the determinantal computations.
The more sophisticated method (C) leverages the observations of \Cref{rmk:HankelStructure} to construct submatrices whose entries belong to polynomial rings with relatively few variables.
This simplifies algebraic manipulations of the resulting minor ideal, such as computing its dimension and determining its roots.
Moreover, it further reduces the vanishing minors for \emph{sparse} polynomials, i.e., those with only a few non-zero coefficients.

We report the results of our computational experiments in \Cref{tab:numtesting}.
The timings heuristically suggest that considering the individual coordinate charts separately is more efficient than computing a general coordinate transformation.
The MAGMA \cite{magma} and Macaulay2 \cite{M2} code for reproducing the listed examples is available at \cite{code}.

\begin{figure}[ht]
\centering

\scalebox{.9}{
\begin{subfigure}{0.3\textwidth}
\centering
\begin{tikzpicture}[scale=0.4]

\foreach \x in {2,5,7,9} {
    \foreach \y in {2,4,7,9} {
        \fill[gray!25] (\x-1,\y-1) rectangle (\x,\y);
    }
}

\draw[gray!50,thin] (1,4) -- (1,10);
\draw[gray!50,thin] (3,7) -- (3,10);
\draw[gray!50,thin] (6,9) -- (6,10);

\draw[gray!50,thin] (0,9) -- (6,9);
\draw[gray!50,thin] (0,7) -- (3,7);
\draw[gray!50,thin] (0,4) -- (1,4);

\draw[thick] (0,0) rectangle (10,10);

\draw[thick]
(1,0)--(1,4)--(3,4)--(3,7)--(6,7)--(6,9)--(10,9);

\coordinate (A) at (1.5,3.5);   
\coordinate (B) at (8.5,1.5);   
\coordinate (C) at (6.5,6.5);   
\coordinate (D) at (4.5,8.5);   

\fill (A) circle (3pt);
\fill (B) circle (3pt);
\fill (C) circle (3pt);
\fill (D) circle (3pt);

\end{tikzpicture}
\caption*{(A)}
\end{subfigure}
\begin{subfigure}{0.3\textwidth}
\centering
\begin{tikzpicture}[scale=0.4]

\foreach \x in {2,4,5,9} {
    \foreach \y in {5,8,9,10} {
        \fill[gray!25] (\x-1,\y-1) rectangle (\x,\y);
    }
}

\draw[gray!50,thin] (1,4) -- (1,10);
\draw[gray!50,thin] (3,7) -- (3,10);
\draw[gray!50,thin] (6,9) -- (6,10);

\draw[gray!50,thin] (0,9) -- (6,9);
\draw[gray!50,thin] (0,7) -- (3,7);
\draw[gray!50,thin] (0,4) -- (1,4);

\draw[thick] (0,0) rectangle (10,10);

\draw[thick]
(1,0)--(1,4)--(3,4)--(3,7)--(6,7)--(6,9)--(10,9);

\coordinate (A) at (8.5,9.5);   
\coordinate (B) at (4.5,7.5);   
\coordinate (C) at (3.5,8.5);   
\coordinate (D) at (1.5,4.5);   

\fill (A) circle (3pt);
\fill (B) circle (3pt);
\fill (C) circle (3pt);
\fill (D) circle (3pt);

\end{tikzpicture}
\caption*{(B)}
\end{subfigure}
\begin{subfigure}{0.3\textwidth}
\centering
\begin{tikzpicture}[scale=0.4]

\foreach \x in {1,2,4,8} {
    \foreach \y in {3,6,8,10} {
        \fill[gray!25] (\x-1,\y-1) rectangle (\x,\y);
    }
}

\draw[gray!50,thin] (1,4) -- (1,10);
\draw[gray!50,thin] (3,7) -- (3,10);
\draw[gray!50,thin] (6,9) -- (6,10);

\draw[gray!50,thin] (0,9) -- (6,9);
\draw[gray!50,thin] (0,7) -- (3,7);
\draw[gray!50,thin] (0,4) -- (1,4);

\draw[thick] (0,0) rectangle (10,10);

\draw[thick]
(1,0)--(1,4)--(3,4)--(3,7)--(6,7)--(6,9)--(10,9);

\coordinate (A) at (0.5,2.5);
\coordinate (B) at (1.5,5.5);
\coordinate (C) at (3.5,7.5);
\coordinate (D) at (7.5,9.5);

\draw[blue!60,thin] (A)--(B)--(C)--(D);

\fill (A) circle (3pt);
\fill (B) circle (3pt);
\fill (C) circle (3pt);
\fill (D) circle (3pt);
\end{tikzpicture}
\caption*{(C)}
\end{subfigure}
}

\captionsetup{width=0.9\textwidth}
\caption{Different methods for extracting minors from $\mathcal{I}_{F,\ell}(\gamma)$: considering random minors (A), selecting block-triangular submatrices (B), and following contraction chains (C).}
  \label{fig:threeApproaches}
\end{figure}

\begin{table}[ht]
\centering
\scalebox{.94}{ \begin{tabular}{| c c c c | c c c | c c c |}
\cline{5-10}
\multicolumn{4}{c|}{} 
& \multicolumn{3}{c|}{Plain} 
& \multicolumn{3}{c|}{Generic chart} \\
\hline
Form & \# supports & local rank & repetitions
& (A) & (B) & (C) 
& (A) & (B) & (C) \\
\hline
$(x_0^2+x_0x_2+x_1^2)x_1$ & 3 & 4 & $\times$100 & 16.05 & 0.82 & \textbf{0.03} & 55.82 & 7.15 & \textbf{5.23} \\
$x_0^2x_1x_2$ & 1 & 4 & $\times$10 & 10.53 & 0.84 & \textbf{0.01} & 14.13 & \textbf{3.34} & \textbf{3.18} \\
$x_0^2(x_1+x_2)+x_1^2(x_0+x_2)+x_2^2(x_0+x_1)$ & 9 & 5 & $\times$20 & 11.18 & \textbf{0.50} & \textbf{0.44} & 60.50 & 25.11 & \textbf{13.74} \\
$x_0^2x_1^2x_2$ & 2 & 6 & $\times$5 & - & 17.98 & \textbf{0.01} & - & 113.35 & \textbf{54.64} \\
$x_0^3 + x_1^3 + x_2^3 + x_0x_3x_2 + x_0x_3^2$ & 9 & 7 & $\times$3 & - & 3.16 & \textbf{0.93} & - & - & - \\
$x_0(x_0^3 + x_0^2x_1 + x_0x_2^2 + x_1^3 + x_2^3 + x_3^3)$ & 1 & 8 & $\times$3 & - & 143.63 & \textbf{21.15} & - & - & -  \\
\hline
\end{tabular} }
\captionsetup{width=\textwidth}
\caption{Numerical results of the computational testing with the methods (A), (B), and (C) discussed above.
For each form, we report the total running time (in seconds) to repeatedly compute the minimal supports (a \emph{repetitions} number of times).
The columns labeled \emph{Plain} correspond to a direct application of \Cref{algo:localgad}, whereas in \emph{Generic chart} it was preceded by a generic change of coordinates. Computations that exceeded the time limit (300\,s) are marked with a dash. } \label{tab:numtesting}
\end{table}


\section{Comparison with previous methods} \label{sec:comparison}

\subsection{Extension-based algorithms}
Sylvester's algorithm for binary forms has been extended to an exact Waring decomposition algorithm \cite{BCMT10}, and further generalized to an algorithm for computing minimal apolar schemes \cite{BT20} and minimal GADs \cite{BMT25}.
These ideas can be further refined for computing minimal \emph{local} apolar schemes \cite{BF25}.
Roughly speaking, since a minimal local apolar scheme at $x_0$ is locally Gorenstein \cite{BB14} with dual generator extending $f_{x_0}$ \cite{BJMR18}, one parametrizes functionals in $(\R')^*$ that agree with $f_{x_0}$ up to degree $d$.
Exploiting the properties of Hankel operators, these objects are determined by imposing the commutation of the matrices representing the multiplication by the variables in the local algebra, upon a previous choice of basis.
See \cite[Section 6]{BT20} for more details.

The approach of this paper is fundamentally different from these methods: instead of parametrizing tensor extensions and computing commuting conditions, which leads to the solution of (possibly many) systems of quadrics in a large number of extension parameters, we focus on minimizing the dimension of the symbolic inverse system.
This basis-free strategy requires solving a typically overdetermined polynomial system in only $n$ variables, although the equations can have high degree.

We compare the two methods with the form $F = x_0^2x_1 + x_0x_1x_2 + x_1^3 \in \S_3$ of \Cref{ex:ex1-C}. 
Following \cite[Algorithm 3]{BT20}, we consider a global functional $\Lambda\in \k[y_1,y_2]^*$ extending $f_{x_0}$, i.e., $\Lambda|_{\leq 3}=f_{x_0}$.
Its Hankel operator $H_{\Lambda}$, defined as 
\[H_\Lambda: \k[y_1,y_2] \to \k[y_1,y_2]^*, \qquad p\mapsto p\star \Lambda,\]
is represented by an infinite matrix, with rows and columns indexed by all monomials in $\k[y_1,y_2]$, such that the $(y^\a,y^\b)$-entry is $\Lambda(y^{\alpha+\beta})$.
We truncate this matrix to rows and columns indexed by monomials of degree at most $\deg F=3$:

\begin{equation} \label{eq:Hlambda}
 H_\Lambda= \hspace{-0.3cm} \scalebox{0.9}{$ \begin{array}{c@{\;}c} & \begin{array}{cccccccccc} \hspace{0.1cm} 1 \hspace{0.25cm} & y_1\hspace{0.45cm} & y_2\hspace{0.45cm} & y_1^2\hspace{0.3cm} & y_1y_2\hspace{0.3cm} & y_2^2\hspace{0.45cm} & y_1^3\hspace{0.3cm} & y_1^2y_2\hspace{0.15cm} & y_1y_2^2 \hspace{0.3cm} & y_2^3 \hspace{0.4cm}
\end{array} \\[6pt] \begin{array}{c} 1\\ y_1\\ y_2\\ y_1^2\\ y_1y_2\\ y_2^2\\ y_1^3\\ y_1^2y_2\\ y_1y_2^2\\ y_2^3 \end{array} & \left( \begin{array}{cccccccccc} 
0 & 2 & 0 & 0 & 1 & 0 & 6 & 0 & 0 & 0 \\ 
2 & 0 & 1 & 6 & 0 & 0 & h_{y_1^4} & h_{y_1^3y_2} & h_{y_1^2y_2^2} & h_{y_1y_2^3} \\ 
0 & 1 & 0 & 0 & 0 & 0 & h_{y_1^3y_2} & h_{y_1^2y_2^2} & h_{y_1y_2^3} & h_{y_2^4} \\ 
0 & 6 & 0 & h_{y_1^4} & h_{y_1^3y_2} & h_{y_1^2y_2^2} & h_{y_1^5} & h_{y_1^4y_2} & h_{y_1^3y_2^2} & h_{y_1^2y_2^3} \\ 
1 & 0 & 0 & h_{y_1^3y_2} & h_{y_1^2y_2^2} & h_{y_1y_2^3} & h_{y_1^4y_2} & h_{y_1^3y_2^2} & h_{y_1^2y_2^3} & h_{y_1y_2^4}\\ 
0 & 0 & 0 & h_{y_1^2y_2^2} & h_{y_1y_2^3} & h_{y_2^4} & h_{y_1^3y_2^2} & h_{y_1^2y_2^3} & h_{y_1y_2^4} & h_{y_2^5}\\ 
6 & h_{y_1^4} & h_{y_1^3y_2} & h_{y_1^5} & h_{y_1^4y_2} & h_{y_1^3y_2^2} & h_{y_1^6} & h_{y_1^5y_2} & h_{y_1^4y_2^2} & h_{y_1^3y_2^3}\\ 
0 & h_{y_1^3y_2} & h_{y_1^2y_2^2} & h_{y_1^4y_2} & h_{y_1^3y_2^2} & h_{y_1^2y_2^3} & h_{y_1^5y_2} & h_{y_1^4y_2^2} & h_{y_1^3y_2^3} & h_{y_1^2y_2^4}\\ 
0 & h_{y_1^2y_2^2} & h_{y_1y_2^3} & h_{y_1^3y_2^2} & h_{y_1^2y_2^3} & h_{y_1y_2^4} & h_{y_1^4y_2^2} & h_{y_1^3y_2^3} & h_{y_1^2y_2^4} & h_{y_1y_2^5}\\ 
0 & h_{y_1y_2^3} & h_{y_2^4} & h_{y_1^2y_2^3} & h_{y_1y_2^4} & h_{y_2^5} & h_{y_1^3y_2^3} & h_{y_1^2y_2^4} & h_{y_1y_2^5} & h_{y_2^6}\\ \end{array} \right). \end{array} $}
\end{equation}

Since $\Lambda$ extends $f_{x_0}$, the entry $\Lambda(y^{\alpha+\beta})$ agrees with the corresponding coefficient of $f_{x_{0}}$ if $|\alpha+\beta|\leq d$, otherwise it is the symbolic moment variable $h_{y^{\alpha+\beta}}$.
Following Line 3 in \cite[Algorithm 3]{BT20}, we need to test (as bases of the local algebras) three complete staircases, namely $\left\{1,y_1,y_2,y_2^2\right\},\, \left\{1,y_1,y_2,y_1y_2\right\}$, and $\left\{1,y_1,y_2,y_1^2\right\}$.
Among them, only $B=\{1,y_1,y_2,y_1y_2\}$ corresponds to a full rank submatrix in the Hankel matrix $H_\Lambda$:
\[H_\Lambda^B = 
 \left(\!\begin{array}{cccc}
      0&2&0&1\\
      2&0&1&0\\
      0&1&0&0\\
      1&0&0&h_{y_1^2y_2^2}
      \end{array}\!\right).\]
The determinant of $H_\Lambda^B$ is nonzero for every value of $h_{y_1^2y_2^2}$, therefore it does not impose any open condition on this moment variable. From the Hankel matrix we compute the multiplication matrices $\mathbb{M}_{y_1}^B$ and $\mathbb{M}_{y_2}^B$ using \cite[Prop. 3.9]{BCMT10}:

\[\resizebox{\textwidth}{!}{$
(\mathbb{M}_{y_1}^B)^t=\left(\!\begin{array}{cccc}
      0&1&0&0\\
      h_{y_1^3y_2}&0&-2\,h_{y_1^3y_2}+6&0\\
      0&0&0&1\\
      2\,h_{y_1^2y_2^2}^{2}+h_{y_1^3y_2^2}&h_{y_1^2y_2^2}&-4\,h_{y_1^2y_2^2}^{2}+h_{y_1^3y_2}-2\,h_{y_1^3y_2^2}&-2\,h_{y_1^2y_2^2}
      \end{array}\!\right),\:(\mathbb{M}_{y_2}^B)^t=\left(\!\begin{array}{cccc}
      0&0&1&0\\
      0&0&0&1\\
      h_{y_1y_2^3}&0&-2\,h_{y_1y_2^3}&0\\
      2\,h_{y_1^2y_2^2}h_{y_1y_2^3}+h_{y_1^2y_2^3}&h_{y_1y_2^3}&-4\,h_{y_1^2y_2^2}h_{y_1y_2^3}+h_{y_1^2y_2^2}-2\,h_{y_1^2y_2^3}&-2\,h_{y_1y_2^3}
      \end{array}\!\right).
      $}\]
We search for symbols to make them commute and, since we only look for local schemes, we also impose that these matrices have a unique eigenvalue (see \cite[Remark 2.9]{BF25}). 
All these conditions define a zero-dimensional polynomial system in $\k[h_{y_1^3y_2},\ldots, h_{y_1^3y_2^2}]$ with two solutions, representing two local schemes of length $4$ apolar to $F$.
After substituting these solutions in $(\mathbb{M}_{y_1}^B)^t$ and $(\mathbb{M}_{y_2}^B)^t$, we find the minimal supports $(0,0)$ and $(0,1)$ from their simultaneous eigenvectors.
These points correspond to the linear forms $x_0$ and $x_0+x_2$, which agree with the minimal supports found in \Cref{ex:ex1-C}. 

However, this method makes understanding local apolar schemes of length $5$ significantly more complicated.
There are $5$ complete staircases to test, and for each one, the commutation of the $5\times 5$ matrices $(\mathbb{M}_{y_1}^B)^t$ and $(\mathbb{M}_{y_2}^B)^t$ is subject to the open condition given by $\det H_\Lambda^B\neq 0$, which now involves moment variables.
On the other hand, this approach can detect apolar schemes that are not evinced by (local) GADs, i.e., those evinced by extensions of $F$ to higher degrees, as discussed in \Cref{rmk:localcactusvsGAD}.
This is precisely the case, as noticed in the following example.

\begin{example} \label{ex:schemenonGAD}
    Let $F = x_0^2x_1 + x_0x_1x_2 + x_1^3 \in \S_3$, which we proved to have cactus rank $4$ (\Cref{ex:ex1-A}), three local GADs of size $4$ (\Cref{rmk:threerank4}) and no local GADs of size $5$ (\Cref{rmk:norank5}).
    However, this form has a local apolar scheme of length $5$, namely that defined by
    \[ I(Z) = (-y_0^2y_2 - 6y_0y_1y_2 + y_1^3, y_1^2y_2, y_2^2), \quad \textnormal{evinced by the local GAD} \quad x_0^0 \left(\frac{1}{3} x_0^3x_1 + \frac{1}{2} x_0^2x_1x_2 + x_0 x_1^3 + \frac{1}{24}x_1^4\right) \in \S_4. \]
    This is completely analogous to \cite[Example 4.6]{BOT24}, except that $Z$ is not minimal for $F$.
\end{example}


\begin{remark}
    \Cref{ex:schemenonGAD} shows that minimal local GADs can be employed to construct small local apolar schemes, but not to prove their (local) minimality.
    Indeed, in its current formulation, \Cref{algo:localgad} cannot detect schemes that are not evinced by a local GAD of a form $F \in \S_d$, which for irredundant schemes occurs precisely when the associated local Artinian Gorenstein algebra has socle degree exceeding $d$.
    By contrast, the algorithms in \cite{BF25, BT20} could detect the minimal apolar schemes even with a larger socle.
    Furthermore, in some cases such as \cite[Example 4.6]{BOT24}, the determinantal approach proposed in this paper is cumbersome, while the previous eigenmethods \cite{BF25, BT20} only involve linear operations.
    Thus, neither method is uniformly superior to the other, but their relative effectiveness depends on the geometry of the minimal GADs/schemes of the given $F \in \S_d$.
    Currently, we have no general criterion to determine a priori which approach will be preferable.
\end{remark}

\subsection{Catalecticant minors} \label{subsec:catalecticant}
The $i$-th catalecticant of $F\in \S_d$, for $0 \leq i \leq d$, is the map
\[ \R_i\to \S_{d-i}, \qquad p \mapsto p\contract F_{\rm dp}. \]
We denote by $\operatorname{Cat}_F^{i}$ the \emph{$i$-th catalecticant matrix}, namely the matrix representing the $i$-th catalecticant of $F$ in the standard monomial bases.
It is immediate to check that $\operatorname{Cat}_F^{i}=(\operatorname{Cat}_F^{d-i})^t$.
The set of polynomials in $\S_d$ whose $i$-th Catalecticant matrix has rank at most $r$ forms an algebraic variety, denoted by $\varUpsilon^{i, d-i}_r(\mathbb PV)$ in \cite{BB14}, which is an instance of a \emph{catalecticant variety} \cite[Def. 1.4]{IK99}. 

We observe that, up to a change of coordinates sending the support of the GAD to $x_0$, the inverse system matrix $\mathcal{I}_{F,x_0}$ contains all $\{\operatorname{Cat}_F^{i}\}_{0 \leq i \leq d}$.
To see that, it suffices to use the fact that for every $G \in \R_i$ and $F\in \S_{d+i}$, the degree-$d$ tails of $(G\contract F)|_{x_0 = 1}$ and $G|_{y_0 = 1} \contract F|_{x_0 = 1}$ are equal \cite[Lemma 2]{BJMR18}.

For example, let $F=x_0^2x_1+x_0x_1x_2+x_1^3\in \S_3$ be the form considered in \Cref{ex:ex1-A}. We have 
\[(\operatorname{Cat}_F^1)^t=\operatorname{Cat}_F^2=\begin{pmatrix}
0 & 2 & 0 & 0 & 1 & 0 \\
2 & 0 & 1 & 6 & 0 & 0 \\
0 & 1 & 0 & 0 & 0 & 0 
\end{pmatrix},\]
and $(\operatorname{Cat}_F^0)^t=\operatorname{Cat}_F^3=(0,2,0,0,1,0,6,0,0,0)$, with respect to the monomial bases in the lexicographic order. 
These catalecticant matrices are all submatrices of
\[\mathcal{I}_{F, x_0}=
\begin{pmatrix}
0 & 2 & 0 & 0 & 1 & 0 & 6 & 0 & 0 & 0 \\
2 & 0 & 1 & 6 & 0 & 0 & 0 & 0 & 0 & 0 \\
0 & 1 & 0 & 0 & 0 & 0 & 0 & 0 & 0 & 0 \\
0 & 6 & 0 & 0 & 0 & 0 & 0 & 0 & 0 & 0 \\
1 & 0 & 0 & 0 & 0 & 0 & 0 & 0 & 0 & 0 \\
0 & 0 & 0 & 0 & 0 & 0 & 0 & 0 & 0 & 0 \\
6 & 0 & 0 & 0 & 0 & 0 & 0 & 0 & 0 & 0 \\
0 & 0 & 0 & 0 & 0 & 0 & 0 & 0 & 0 & 0 \\
0 & 0 & 0 & 0 & 0 & 0 & 0 & 0 & 0 & 0 \\
0 & 0 & 0 & 0 & 0 & 0 & 0 & 0 & 0 & 0
\end{pmatrix}.
\]
Consequently, when $F$ has local GAD-rank $r$, every $(r+1) \times (r+1)$ minor of any $\operatorname{Cat}_F^i$ vanishes.
However, in this example, the local GAD-rank of $F$ is $4$, while the largest catalecticant matrices are $\operatorname{Cat}_F^1$ and $\operatorname{Cat}_F^2$, which have both rank $3$.
Therefore, we cannot deduce the local GAD-rank of $F$ by looking at its catalecticant matrices separately.

We also note that the Hankel matrix $H_\Lambda$ in \cref{eq:Hlambda} coincides with $\mathcal{I}_{F,x_0}$, if we substitute all the moment variables $h_{y^\alpha}$ by zero.
This is a general fact, since the $\star$-action agrees with contraction (see \Cref{prop:equivgenerators}), 
and the infinite operator tail arising from the exponential term in \cref{eq: exponential} vanishes, since the local algebra is supported at zero ($\ell = x_0$).


\section{Conclusions and future directions} \label{sec:conclusion}

We have introduced an independent and complementary perspective to existing methods for the study of minimal local GADs, highlighting concrete cases where this new method allows for a finer analysis of the associated local algebras.
This determinantal method, based on symbolic inverse systems, reduces the computation of minimal supports to the symbolic rank minimization of a Hankel matrix, for which we designed and experimentally tested different routines.
We believe these minor-selection techniques could be further refined and optimized for special choices of $F \in \S_d$, which appears to be a promising avenue for future research to compute minimal GADs of forms with larger parameters $n$ and $d$.

In \Cref{subsec:catalecticant}, we observed that catalecticant matrices may be too small to recover the local GAD-rank of a form $F$.
However, when this rank is sufficiently small compared to $d$, these matrices should contain enough information to 
determine the local GAD-rank of $F$.
We believe it is worth investigating how to find these conditions explicitly, inspired by the similar relation between cactus and catalecticant varieties found in \cite[Thm. 1.7]{BB14}.

Another promising direction for future research is to adapt the present method to construct local schemes with a prescribed Hilbert function, similarly to \cite{BJMR18}, or more refined invariants such as a specified \emph{Jordan Degree Type} \cite{IMM22}.
We note that the Hilbert functions of both the Artinian local algebra (where the graded structure is given by the filtration of powers of the maximal ideal) and the projective punctual scheme (defined by the homogenization of the affine ideal) can be understood solely in terms of the matrix $\mathcal{I}_{F,x_0}$ (see, e.g., \cite{CI12} for the relationship between the local inverse system $\R'\contract f_{\ell}$ and the Hilbert function of the defining scheme).

Finally, it would be natural to investigate whether intrinsic constraints govern the distribution of local GAD-ranks. More precisely, given a finite sequence $(a_1, a_2, \ldots, a_m) \subset \N \cup \{\infty\}$, with $a_1 \in \{0,1\}$, one may ask whether there are positive integers $n,d$ and a form $F \in \S_d$ such that, for each $k$, the number of local GADs of $F$ of size $k$ is exactly $a_k$.

\subsection*{Acknowledgements}
We thank A. Iarrobino and the organizers of \emph{Jordan Types of Artinian Algebras and Geometry of Punctual Hilbert Schemes}, held in celebration of his 80th birthday.
We are grateful to E. Barrilli, A. Bernardi, and B. Mourrain for useful discussions, and to N. Vannieuwenhoven for advice on symbolic rank minimization.
We also thank the anonymous referees for their comments, which led to a significant improvement of the manuscript.

This work has been supported by European Union’s HORIZON–MSCA2023-DN-JD programme under the Horizon Europe (HORIZON) Marie Sklodowska-Curie Actions, grant agreement 101120296 (TENORS), by the Research Foundation - Flanders (FWO: 12ZZC23N), and by the BOF project C16/21/002 by the Internal Funds
KU Leuven.


\begin{thebibliography}{ccccc}

\bibitem{BB12} E. Ballico and A. Bernardi,
\emph{Decomposition of homogeneous polynomials with low rank},
Math. Z. 271, 2012, pp. 1141--1149.

\bibitem{BMT25} E. Barrilli, B. Mourrain, and D. Taufer,
\emph{Generalized Additive Decompositions of Symmetric Tensors},
arXiv:2510.25681, 2025.

\bibitem{BBM14} A. Bernardi, J. Brachat, and B. Mourrain,
\emph{A comparison of different notions of ranks of symmetric tensors},
Linear Algebra Appl. 460, 2014, pp. 205--230.

\bibitem{BCGI07} A. Bernardi, M. V. Catalisano, A. Gimigliano, and M. Idà,
\emph{Osculating varieties of Veronese varieties and their higher secant varieties},
Canad. J. Math. 59, 2007, pp. 488--502.

\bibitem{BF25} A. Bernardi and O. Reig Fité,
\emph{A refinement on the local cactus rank algorithm},
arXiv:2508.15062, 2025.

\bibitem{BJMR18} A. Bernardi, J. Jelisiejew, P. Macias Marques, and K. Ranestad,
\emph{On polynomials with given Hilbert function and applications},
Collect. Math. 69, 2018, pp. 39--64.

\bibitem{BOT24} A. Bernardi, A. Oneto, and D. Taufer,
\emph{On schemes evinced by generalized additive decompositions and their regularity},
J. Math. Pures Appl. 188, 2024, pp. 446--469.
    
\bibitem{BT20} A. Bernardi and D. Taufer,
\emph{Waring, tangential and cactus decompositions},
J. Math. Pures Appl. 143, 2020, pp. 1--30.

\bibitem{magma} W. Bosma, J. Cannon, and C. Playoust,
\emph{The Magma algebra system. I. The user language},
J. Symb. Comput. 24(3-4), 1997, pp. 235--265.

\bibitem{BCMT10} J. Brachat, P. Comon, B. Mourrain, and E. Tsigaridas, 
\emph{Symmetric tensor decomposition},
Linear Algebra Appl. 433, 2010, pp. 1851--1872.

\bibitem{BB14} W. Buczyńska and J. Buczyński,
\emph{Secant varieties to high degree Veronese reembeddings, catalecticant matrices and smoothable Gorenstein schemes},
J. Algebraic Geom. 23(1), 2014, pp. 63--90.

\bibitem{BK24} J. Buczyński and H. Keneshlou,
\emph{Cactus scheme, catalecticant minors, and scheme theoretic equations},
arXiv:2410.21908, 2024.

\bibitem{BR13} A. Bernardi and K. Ranestad, \emph{On the cactus rank of cubic forms}, Journal of Symbolic
Computation 50 (2013), 291–297.

\bibitem{C05} E. Carlini,
\emph{Reducing the number of variables of a polynomial},
in Algebraic Geometry and Geometric Modeling, Springer, 2006, pp. 237--247.

\bibitem{CM23} A. Casarotti and M. Mella,
\emph{From non-defectivity to identifiability},
J. Eur. Math. Soc. 25(3), 2023, pp. 913--931.

\bibitem{COV14} L. Chiantini, G. Ottaviani, and N. Vannieuwenhoven,
\emph{An algorithm for generic and low-rank specific identifiability of complex tensors},
SIAM J. Matrix Anal. Appl. 35, 2014, pp. 1265--1287.

\bibitem{CI12} Y. H. Cho and A. Iarrobino,
\emph{Inverse systems of zero-dimensional schemes in $\P^n$},
J. Algebra 366, 2012, pp. 42--77.

\bibitem{CLOS07} D. Cox, J. Little, and D. O'Shea,
\emph{Ideals, Varieties, and Algorithms: An Introduction to Computational Algebraic Geometry and Commutative Algebra},
Springer, 2007.


\bibitem{G96} A. V. Geramita, 
\emph{Inverse systems of fat points: Waring’s problem, secant varieties of Veronese varieties and parameter spaces for Gorenstein ideals},
Queen’s Papers in Pure and Appl. Math. 102, 1996, pp. 2-114.

\bibitem{GKT25} F. Gesmundo, L. Kayser, and S. Telen,
\emph{Hilbert functions of chopped ideals},
J. Algebra 666, 2025, pp. 415--445.

\bibitem{M2} D. R. Grayson and M. E. Stillman,
\emph{Macaulay2, a software system for research in algebraic geometry},
Available at \url{http://www2.macaulay2.com}.
          
\bibitem{Iar84} A. Iarrobino,
\emph{Compressed algebras: Artin algebras having given socle degrees and maximal length},
Trans. Amer. Math. Soc., 285(1), 1984, pp. 337--378.

\bibitem{IE78} A. Iarrobino and J. Emsalem,
\emph{Some zero-dimensional generic singularities; finite algebras having small tangent space},
Compositio Math. 36(2), 1978, pp. 145--188.

\bibitem{IK99} A. Iarrobino and V. Kanev,
\emph{Power sums, Gorenstein algebras, and determinantal loci},
Springer Science \& Business Media, 1999. 

\bibitem{IMM22} A. Iarrobino, P. Macias Marques, and C. McDaniel,
\emph{Artinian algebras and Jordan type},
J. Commut. Algebra 14(3), 2022, pp. 365--414.

\bibitem{L11} J. M. Landsberg, 
\emph{Tensors: Geometry and Applications},
Graduate Studies in Mathematics 128, American Mathematical Society, 2012.

\bibitem{LKS25} J. Lindberg, J. Kileel, and B. Shi, 
\emph{Efficient Tensor Decomposition via Moment Matrix Extension},
J. Symb. Comput. 139, 2027.

\bibitem{M19} F. S. Macaulay,
\emph{The algebraic theory of modular systems},
in Cambridge tracts in mathematics and mathematical physics 19, Cambridge University Press, 1916. 

\bibitem{M18} B. Mourrain,
\emph{Polynomial-exponential decomposition from moments},
Found. Comput. Math. 18, 2018, pp. 1435–1492.

\bibitem{MPT25} M. Mula, F. Pintore, and D. Taufer,
\emph{The Hessian of elliptic curves as a Lattès map},
Math. Comp., 2026.


\bibitem{code} Code repository: \href{https://github.com/DTaufer/MinimalLocalGAD/}{github.com/DTaufer/MinimalLocalGAD/}
\end{thebibliography}

\end{document}